\documentclass[preprint,12pt,1p]{elsarticle}

\usepackage{graphicx}
\usepackage[all,cmtip]{xy}
\usepackage{amsthm,amsmath}
\usepackage{dsfont}
\usepackage{mathtools,amssymb}
\usepackage{extarrows}
\usepackage{tikz-cd}
\journal{\textnormal{$\cdots$}}
\usepackage{graphicx}
\usepackage{amsfonts}

\usepackage[colorlinks=true,linkcolor=blue,citecolor={blue}]{hyperref}

\theoremstyle{plain}
\newtheorem{theorem}{Theorem}[section]
\newtheorem{proposition}[theorem]{Proposition}
\newtheorem{lemma}[theorem]{Lemma}
\newtheorem{corollary}[theorem]{Corollary}

\theoremstyle{definition}

\newtheorem{example}[theorem]{Example}

\newcommand{\ro}{\circ}
\newcommand{\vecf}{\mathfrak{X}}
\newcommand{\met}{\langle\cdot\,,\cdot\rangle}
\newcommand{\rel}{\mathbb{R}}
\newcommand{\intg}{\mathds{Z}}

\newcommand{\cplx}{\mathds{C}}
\newcommand{\n}{\nabla}
\newcommand{\nc}{\n^{0,1}}
\newcommand{\ns}{\n^\mathbf{s}}

\newcommand{\w}{\omega}
\newcommand{\et}{\quad \textnormal{and}\quad}
\newcommand{\id}{\operatorname{Id}}

\newcommand{\el}{\operatorname{L}}
\newcommand{\tr}{\operatorname{tr}}
\newcommand{\A}{\mathbf{A}}

\newcommand{\End}{\operatorname{End}}

\newcommand{\g}{\mathfrak{g}}
\newcommand{\h}{\mathfrak{h}}
\newcommand{\m}{\mathfrak{m}}
\newcommand{\Ad}{\operatorname{Ad}}
\newcommand{\ad}{\operatorname{ad}}
\newcommand{\Span}{\operatorname{span}}
\newcommand{\Aut}{\operatorname{Aut}}

\newcommand{\SL}{\operatorname{SL}}
\newcommand{\T}{\operatorname{T}}
\newcommand{\lN}{\operatorname{L}^\n}
\newcommand{\Ll}{\operatorname{L}}

\newcommand{\la}{\operatorname{L}^{0,1}}
\newcommand{\ra}{\operatorname{R}^{0,1}}
\newcommand{\ls}{\operatorname{L}^\mathbf{s}}
\newcommand{\rs}{\operatorname{R}^\mathbf{s}}
\newcommand{\lm}{\operatorname{L}^{1,0}}
\newcommand{\ric}{\operatorname{ric}}
\newcommand{\LS}{\operatorname{L}}

\newcommand{\F}{\operatorname{W}}
\newcommand{\J}{\operatorname{J}}
\numberwithin{equation}{section}
\newcommand{\nlp}{\mathfrak{n}}
\newcommand{\U}{\operatorname{U}}
\newcommand{\I}{\operatorname{I}}
\newcommand{\SU}{\operatorname{SU}}
\newcommand{\Aa}{x_1}
\newcommand{\Ba}{x_2}
\newcommand{\Ab}{x_3}
\newcommand{\Bb}{x_4}
\newcommand{\Ac}{x_5}
\newcommand{\Bc}{x_6}
\newcommand{\Ca}{y_1}
\newcommand{\Da}{y_2}
\newcommand{\Cb}{y_3}
\newcommand{\Db}{y_4}
\newcommand{\Cc}{y_5}
\newcommand{\Dc}{y_6}

\newcommand{\aj}{\divideontimes}
\newcommand{\Kc}{K^{0,1}}
\newcommand{\Ks}{K^\mathbf{s}}
\newcommand{\nl}{\mathfrak{n}}
\newcommand{\lab}{\operatorname{L}^{a,b}}
\newcommand{\nab}{\n^{a,b}}
\newcommand{\ab}{(a,b)}
\newcommand{\somjk}{\displaystyle\sum_{j=1}^{k}}
\newcommand{\somjl}{\displaystyle\sum_{j=1}^{\ell}}

\begin{document}

\begin{frontmatter}

\title{\textbf{Some invariant connections on symplectic reductive homogeneous spaces}\vskip 0.5cm \today}

\author[l1]{Abdelhak Abouqateb}
\ead{a.abouqateb@uca.ac.ma}
\author[l1]{Othmane Dani}
\ead{othmanedani@gmail.com}
\address[l1]{Department of Mathematics, Faculty of Sciences and Technologies, Cadi Ayyad University  B.P.549 Gueliz Marrakesh, Morocco}
	
\begin{abstract}
A symplectic reductive homogeneous space is a pair $(G/H,\Omega)$, where $G/H$ is a reductive homogeneous $G$-space and $\Omega$ is a $G$-invariant symplectic form on it. The main examples include symplectic Lie groups, symplectic symmetric spaces, and flag manifolds. This paper focuses on the existence of a \textit{natural} symplectic connection on $(G/H,\Omega)$. First, we introduce a family $\{\nab\}_{\ab\in\rel^2}$ of $G$-invariant connection on $G/H$, and establish that $\n^{0,1}$ is flat if and only if $(G/H,\Omega)$ is locally a symplectic Lie group. Next, we show that among all $\{\nab\}_{\ab\in\rel^2}$, there exists a unique symplectic connection, denoted by $\ns$, corresponding to $a=b=\tfrac{1}{3}$, a fact that seems to have previously gone unnoticed. We then compute its curvature and Ricci curvature tensors. Finally, we demonstrate that the $\SU(3)$-invariant preferred symplectic connection of the Wallach flag manifold $\SU(3)/\mathbb{T}^2$, introduced in the Cahen-Gutt-Rawnsley paper \cite{Cah1} (Lett. Math. Phys. 48 (1999), 353–364), coincides with the natural symplectic connection $\ns$, which is furthermore Ricci-parallel.
\end{abstract}

\end{frontmatter}	
	
\section{Introduction and main results}

A \textit{Fedosov manifold} is a triple $(M,\Omega,\n)$ where $(M,\Omega)$ is a symplectic manifold and $\n$ is a symplectic connection, i.e., $T^\n=0$ and $\n\Omega=0$. This notion was introduced first in \cite{Gel}. Of course, any symplectic manifold $(M,\Omega)$ can be endowed with a Fedosov structure, since symplectic connections always exist on symplectic manifolds. Nevertheless, in general there is no natural choice of such a structure. Hence, the simplest question that one may ask is the following (cf. \cite{Cah}): When is there a natural choice of a symplectic connection on $(M,\Omega)$? To have an answer we need to provide $(M,\Omega)$ with some extra structure. For example, if $(M,\Omega)$ admits an integrable almost complex structure $\J:TM\to TM,\,\J^2=-\id_{TM}$, which is compatible with $\Omega$, then we can choose the Levi-Civita connection associated to the Riemannian metric induced by $\Omega$ and $\J$. On the other hand, it was shown in \cite{Bie} that on any symplectic symmetric space, there is a natural choice of a symplectic connection. Moreover, in \cite{BenBo6} it was introduced a natural symplectic connection on any symplectic Lie group. Both classes (symplectic symmetric spaces and symplectic Lie groups) are particular examples of symplectic reductive homogeneous spaces which is a more general class that includes other things like flag manifolds (cf. \cite{Alek,Arv}). The spirit of this paper is to give an answer of the above question in this general class, i.e., in the case where $(M,\Omega)$ is a symplectic reductive homogeneous space.

A \textit{symplectic homogeneous space} $(G/H,\Omega)$ is a homogeneous $G$-space $G/H$ endowed with a $G$-invariant symplectic form $\Omega$. Symplectic homogeneous spaces are very interesting in their own, and they are extensively studied by many authors (see for instance \cite{Chu,Bo}). As mentioned above, in this work, we are interested in a subclass of symplectic homogeneous spaces $(G/H,\Omega)$ for which $G/H$ is reductive, i.e., there exists a vector subspace $\mathfrak{m}$ of the Lie algebra $\mathfrak{g}$ of $G$ such that $\mathfrak{g}=\mathfrak{h}\oplus\mathfrak{m}$ and $\Ad(H)(\mathfrak{m})=\mathfrak{m}$, where $\mathfrak{h}$ is the Lie algebra of $H$. Such a pair $(G/H,\Omega)$ will be called a \textit{symplectic reductive homogeneous $G$-space}.\\

Let $(G/H,\Omega)$ be a symplectic reductive homogeneous $G$-space, $\mathfrak{g}=\mathfrak{h}\oplus\mathfrak{m}$ a fixed reductive decomposition, and $\omega\in\bigwedge^2\mathfrak{m}^*$ the symplectic tensor induced by $\Omega$ (see Section \ref{Prelim} for notations). Since $\omega$ is nondegenerate, for any two real numbers $a,b\in\rel$, we can define a product $\lab:\mathfrak{m}\times\mathfrak{m}\to\mathfrak{m}$ on $\mathfrak{m}$ as follows:
\begin{equation}\label{na}
\omega\left(\lab_x(y),z\right):=a\,\omega\left([x,y]_\mathfrak{m},z\right)+b\,\omega\left([x,z]_\mathfrak{m},y\right).
\end{equation}

It is obvious that if $G/H$ is \textit{locally symmetric}, meaning, $[\m,\m]\subseteq\h$, then the product $\lab$ is just the trivial product on $\m$ for any $a,b\in\rel$. Consequently, $\lab$ defines a unique torsion-free $G$-invariant connection $\n^0$ on $G/H$ (cf. \cite{Nom1}). Furthermore, since $\Omega$ is $G$-invariant, it follows from \cite[pp. 193-194]{Kob2} that $\n^0$ is symplectic. Our first main result explores the properties of the products $\lab$ in the case where $G/H$ is not locally symmetric, i.e., $[\m,\m]\nsubseteq\h$. Specifically, we establish the following:

\begin{theorem}\label{thm1}
Let $(G/H,\Omega)$ be a symplectic reductive non-locally symmetric homogeneous $G$-space, and $\g=\h\oplus\m$ a fixed reductive decomposition. Then, for any $a,b\in\rel$, the product $\lab:\mathfrak{m}\times\mathfrak{m}\to\mathfrak{m}$ defined by \eqref{na} gives rise to a $G$-invariant connection $\nab$ on $G/H$. Moreover, \begin{itemize}
\item[$(1)$] $\nab$ preserves the symplectic form $\Omega$ if and only if $\,a=b$; 
\item[$(2)$] $\nab$  is torsion-free  if and only if $\,a=\tfrac{1-b}{2}$;  
\item[$(3)$] $\nab$ is symplectic if and only if $\,a=b=\frac{1}{3}$.
\end{itemize}
\end{theorem}

By Theorem \ref{thm1}, it is evident that the torsion-free $G$-invariant connection $\n^{0,1}$, corresponding to the parameters $a=0$ and $b=1$, generalizes the canonical flat connection of a symplectic Lie group. Given this distinctive property, we shall refer to it as the \textit{Zero-One connection}. Our second main result presents notable algebraic and geometric results concerning connected symplectic reductive homogeneous spaces with flat Zero-One connection.

\begin{theorem}\label{Theo1Flat1}
Let $(G/H,\Omega)$ be a connected symplectic reductive homogeneous $G$-space, and $\g=\h\oplus\m$ a fixed reductive decomposition. Assume that $G$ acts almost effectively on $G/H$, and the Zero-One connection $\n^{0,1}$ is flat. Then, the subgroup $G_\m:=\langle\exp_G(\m)\rangle\subset G$ generated by $\m$ satisfies:
\begin{itemize}
\item[$(1)$] $G_\m$ is a connected normal (immersed) Lie subgroup of $G$. Moreover, $(G_\m,\imath_{G_\m}^*\pi_G^*\Omega)$ is a symplectic Lie group, where $\pi_G:G\to G/H$ is the canonical projection, and $\imath_{G_\m}:G_\m\hookrightarrow G$ is the canonical injection;
\item[$(2)$] $G_\m$ acts transitively on $G/H$, and preserving both the symplectic form $\Omega$ and the Zero-One connection $\n^{0,1}$. Moreover, the isotropy group at $H\in G/H$ is equal to the discrete Lie subgroup $\Gamma:=G_\m\cap H$ of $G_\m$. In particular, $(G/H,\Omega)$ is locally a symplectic Lie group;
\item[$(3)$] If $G/H$ is simply connected, then $\pi_{G_\m}:(G_\m,\pi_{G_\m}^*\Omega)\to(G_\m/\Gamma,\Omega)$ is a symplectomorphism, where $\pi_{G_\m}:G_\m\to G_\m/\Gamma$ is the canonical projection. In particular, $(G/H,\Omega)$ is a symplectic Lie group;
\item[$(4)$] If $G/H$ is compact, then $G_\m$ is abelian and hence $(G/H,\Omega)$ is a symplectic torus.
\end{itemize} 
\end{theorem}

Invoking Theorem \ref{thm1} once again, we refer to the symplectic connection $\n^{\frac{1}{3},\frac{1}{3}}$ as the \textit{natural symplectic connection} of $(G/H,\Omega)$, denoted by $\ns$. Our third main result is:

\begin{theorem}\label{ThmSympNilman}
Let $(G/H,\Omega)$ be a connected symplectic reductive homogeneous $G$-space, and $\g=\h\oplus\m$ a fixed reductive decomposition. Assume that $G$ acts almost effectively on $G/H$, and $\n^{0,1},\,\n^\mathbf{s}$ are both flat. Then, $(G/H,\Omega)$ is a symplectic nilmanifold.
\end{theorem}
 
A symplectic connection $\n$ on a symplectic manifold $(M,\Omega)$ is called \textit{preferred} if its Ricci curvature tensor $\ric^{\n}$ satisfies 
$$\left(\n_{X}\ric^{\n}\right)\left(Y,Z\right)+\left(\n_{Z}\ric^{\n}\right)\left(X,Y\right)+\left(\n_{Y}\ric^{\n}\right)\left(Z,X\right)=0,$$ for any vector fields $X,Y$ and $Z\in\vecf(M)$.
In \cite[pp. $362$-$363$]{Cah1} it was shown that the Wallach flag manifold $\SU(3)/\mathbb{T}^2$ admits a unique
$\SU(3)$-invariant preferred symplectic connection; however, the explicit expression of this connection was not provided. Our fourth main result shows that:

\begin{theorem}\label{SU3Pref}
The $\SU(3)$-invariant preferred symplectic connection of the Wallach flag manifold $\SU(3)/\mathbb{T}^2$ coincides with the natural symplectic connection $\ns$. Moreover, it is Ricci-parallel.
\end{theorem}

\section{A family of invariant connections on symplectic reductive non-locally symmetric homogeneous spaces}\label{Prelim}

In this section, we review essential and well-known facts that will be employed later, then we prove our first main result.

Let $(G/H,\Omega)$ be a symplectic reductive homogeneous $G$-space, and  $\mathfrak{g}=\mathfrak{h}\oplus\mathfrak{m}$ a fixed reductive decomposition. A \textit{symplectic tensor}  $\omega$ on $\m$ is a nondegenerate, skew-symmetric bilinear form $\omega:\mathfrak{m}\times\mathfrak{m}\to{\mathbb{R}}$. A   
symplectic tensor  $\omega\in\bigwedge^2\m^*$ is called $\Ad(H)$-\textit{invariant} if 
\begin{equation}\label{OmgSmlInv}
\omega\big(\Ad_h(x),\Ad_h(y)\big)=\omega(x,y),\qquad\forall\,x,y\in\mathfrak{m},\,\forall\,h\in H.
\end{equation}
It is well known (cf. \cite{Chu,Bo}) that for a homogeneous $G$-space $G/H$, there is a one-to-one correspondence between $G$-invariant symplectic forms on $G/H$ and $\Ad(H)$-invariant $2$-cocycles of $\mathfrak{g}$ which have $\h$ as radical. Thus, if $G/H$ is a reductive homogeneous $G$-space with a fixed reductive decomposition $\mathfrak{g}=\mathfrak{h}\oplus\mathfrak{m}$, and if we denote by $x_\mathfrak{m}$ the projection of $x$ on $\mathfrak{m}$, then we have a one-to-one correspondence between $G$-invariant symplectic forms on $G/H$ and $\Ad(H)$-invariant symplectic tensors $\omega$ of $\mathfrak{m}$ which satisfy:
\begin{equation}\label{Cocy}
\omega\left([x,y]_\mathfrak{m},z\right)+\omega\left([z,x]_\mathfrak{m},y\right)+\omega\left([y,z]_\mathfrak{m},x\right)=0,\qquad\forall\,x,y,z\in\m.
\end{equation}
Since any Lie group $G$ can be considered as a reductive homogeneous $G$-space, where $H=\{e_G\}$ is the identity element of $G$ and $\mathfrak{m}=\mathfrak{g}$, one can easily see that there is a one-to-one correspondence between left-invariant symplectic forms on $G$ and $2$-cocycle symplectic tensors of $\mathfrak{g}$.\\

Hereafter, and unless the contrary is explecitly stated, $(G/H,\Omega)$ will be a symplectic reductive non-locally symmetric homogeneous $G$-space, $\mathfrak{g}=\mathfrak{h}\oplus\mathfrak{m}$ a fixed reductive decomposition, $\omega\in\bigwedge^2\m^*$ the symplectic tensor corresponding to $\Omega$, and we denote by $x_\mathfrak{m}$ (resp. $x_\h$) the projection of $x\in\g$ on $\mathfrak{m}$ (resp. on $\h$). Additionally, for any $x\in\mathfrak{m}$, we define a linear endomorphism $\lm_x:\mathfrak{m}\to\mathfrak{m}$ by $\lm_x(y):=[x,y]_\mathfrak{m}$, and we denote by $u^\mathfrak{m}$ the unique vector of $\m$ that satisfies $\omega\left(u^\mathfrak{m},x\right)=\tr\big(\el_x^{1,0})$ .\\

Now, since $G/H$ is a reductive homogeneous $G$-space, then (cf. \cite{Nom1}) there is a one-to-one correspondence between the set of $G$-invariant connections on $G/H$ and the set of products $\Ll:\mathfrak{m}\times\mathfrak{m}\rightarrow\mathfrak{m}$ on $\m$ which are $\Ad(H)$-invariant. In fact, if $\n$ is a $G$-invariant connection on $G/H$, then under the identification of $T_H(G/H)$ with $\mathfrak{m}$, its associated product $\Ll^\n:\mathfrak{m}\times\mathfrak{m}\rightarrow\mathfrak{m}$ is given by:
\begin{equation}
\Ll^\n_x(y)=\A^\n_{x^\#}y^\#_H,
\end{equation}    
for $x,y\in\mathfrak{m}$, where $x^\#,y^\#$ are the fundamental vector fields associated to $x,y$, and $\A^\n_{x^\#}:=\n_{x^\#}-\mathcal{L}_{x^\#}$ is the Nomizu's operator related to $\n$ and $x^\#$. It is obvious that the torsion $T^\n$ and the curvature $K^\n$ tensor fields of $\n$ are also $G$-invariant. Thus, they are completely determined by their value at the origin $H\in G/H$, and therefore they can be expressed as follows:
\begin{eqnarray}
T^\n(x,y)&=&\lN_x(y)-\lN_y(x)-[x,y]_\m;\label{Tor}\\
K^\n(x,y)&=&\left[\lN_x,\lN_y\right]-\lN_{[x,y]_\mathfrak{m}}-\ad_{[x,y]_\h}.\label{Curv}
\end{eqnarray}
Furthermore, if $\n$ is a $G$-invariant connection on $G/H$, then we can define a $(1,2)$-tensor field $\mathcal{N}^\n$ by the following:
\begin{equation}
\Omega\left(\mathcal{N}^\n(x^\#,y^\#),z^\#\right):=\left(\n_{x^\#}\Omega\right)(y^\#,z^\#),
\end{equation}
for $x,y,z\in\g$. It is clear that $\n$ preserves $\Omega$ if and only if $\mathcal{N}^\n=0$. On the other hand, since $\Omega$ and $\n$ are both $G$-invariant, it follows that $\mathcal{N}^\n$ is also $G$-invariant and hence it is completely determined by its value at $H\in G/H$. Thus, for $x,y,z\in\m$, $\mathcal{N}^\n$ is given by:
\begin{eqnarray}\label{NGinv}
\omega\left(\mathcal{N}^\n(x,y),z\right)&=&\Omega\left(\mathcal{N}^\n(x^\#,y^\#),z^\#\right)(H)\nonumber\\
&=&\left(\n_{x^\#}\Omega\right)(y^\#,z^\#)(H)\nonumber\\
&=&x^\#\Big(\Omega(y^\#,z^\#)\Big)(H)-\Omega\left(\n_{x^\#}y^\#,z^\#\right)(H)-\Omega\left(y^\#,\n_{x^\#}z^\#\right)(H)\nonumber\\
&=&-\Omega\left(\A^\n_{x^\#}y^\#,z^\#\right)(H)-\Omega\left(y^\#,\A^\n_{x^\#}z^\#\right)(H)\nonumber\\
&=&-\omega\left(\lN_x(y),z\right)-\omega\left(y,\lN_x(z)\right).
\end{eqnarray}
As a result, $\n$ preserves $\Omega$ if and only if $\lN_x:\m\to\m$ is skew-symmetric with respect to $\omega$, for any $x\in\m$.\\

According to the discussion above, the proof of Theorem \ref{thm1} is a direct consequence of the following proposition.

\begin{proposition}\label{PrpPrda}
For any $\ab\in\rel^2$, the product $\lab:\mathfrak{m}\times\mathfrak{m}\to\mathfrak{m}$ introduced in \eqref{na} satisfies the following properties:
\begin{itemize}
\item[$1.$] $\lab$ is $\Ad(H)$-invariant, i.e., $\Ad(H)\subseteq\Aut\left(\mathfrak{m},\lab\right)$;
\item[$2.$] $\lab_x(y)-\lab_y(x)=(2a+b)[x,y]_\mathfrak{m}$,\,  for $x,y\in\mathfrak{m}$;
\item[$3.$] $\omega\left(\mathcal{N}^{a,b}(x,y),z\right)=(b-a)\,\w(x,[y,z]_\m)$,\, for $x,y,z\in\m$.
\end{itemize}
\end{proposition}

\begin{proof}
Let $x,y,z\in\mathfrak{m}$, and $h\in H$. For the first assertion, we have:
\begin{eqnarray*}
\omega\left(\lab_{\Ad_h(x)}\left(\Ad_h(y)\right),z\right)&=&a\,\omega\big([\Ad_h(x),\Ad_h(y)]_\mathfrak{m},z\big)+b\,\omega\big([\Ad_h(x),z]_\mathfrak{m},\Ad_h(y)\big)\\
&=&a\,\omega\big(\Ad_h([x,y]_\mathfrak{m}),z\big)+b\,\omega\big(\Ad_h([x,\Ad_{h^{-1}}(z)]_\mathfrak{m}),\Ad_h(y)\big)\\
&\stackrel{\rm\eqref{OmgSmlInv}}{=}&a\,\omega\big([x,y]_\mathfrak{m},\Ad_{h^{-1}}(z)\big)+b\,\omega\big([x,\Ad_{h^{-1}}(z)]_\mathfrak{m}),y\big)\\
&=&\omega\big(\lab_x(y),\Ad_{h^{-1}}(z)\big)\\
&\stackrel{\rm\eqref{OmgSmlInv}}{=}&\omega\big(\Ad_{h}\left(\lab_x(y)\right),z\big).
\end{eqnarray*}
For the second assertion, we have:
\begin{eqnarray*}
\omega\big(\lab_x(y)-\lab_y(x),z\big)&=&a\,\omega\left([x,y]_\mathfrak{m},z\right)+b\,\omega\left([x,z]_\mathfrak{m},y\right)\\
&&-a\,\omega\left([y,x]_\mathfrak{m},z\right)-b\,\omega\left([y,z]_\mathfrak{m},x\right)\\
&=&2a\,\omega\left([x,y]_\mathfrak{m},z\right)+b\,\omega\left([x,z]_\mathfrak{m},y\right)+b\,\omega\left([z,y]_\mathfrak{m},x\right)\\
&\stackrel{\rm\eqref{Cocy}}{=}&(2a+b)\,\omega\left([x,y]_\mathfrak{m},z\right).
\end{eqnarray*}
For the last assertion, we have:
\begin{eqnarray*}
\omega\left(\mathcal{N}^{a,b}(x,y),z\right)&\stackrel{\rm\eqref{NGinv}}{=}&\omega\left(\lab_x(z),y\right)-\omega\left(\lab_x(y),z\right)\\
&=&a\,\omega\left([x,z]_\mathfrak{m},y\right)+b\,\omega\left([x,y]_\mathfrak{m},z\right)\\
&&-a\,\omega\left([x,y]_\mathfrak{m},z\right)-b\,\omega\left([x,z]_\mathfrak{m},y\right)\\
&=&(b-a)\,\w([x,y]_\m,z)+(b-a)\,\w([z,x]_\m,y)\\
&\stackrel{\rm\eqref{Cocy}}{=}&(b-a)\,\w(x,[y,z]_\m).
\end{eqnarray*}
This finishes the proof.
\end{proof}

Note that for certain specific values of $\ab$, we recover some well-known $G$-invariant connections. For example, when $\ab=(\tfrac{1}{2},0)$ $\big($resp. $\ab=(0,0)\big)$, the corresponding $G$-invariant connection is the canonical connection of the first (resp. second) kind introduced in \cite{Nom1}.\\

A symplectic reductive homogeneous $G$-space $(G/H,\Omega)$ is called a \textit{symplectic solvmanifold} if $G$ is a connected solvable Lie group and $H$ is a discrete subgroup of $G$. In such a case, we denote by $\mathcal{R}$ (resp. $\Gamma$) instead of $G$ (resp. $H$) the connected solvable Lie group (resp. the discrete subgroup). If $(\mathcal{R}/\Gamma,\Omega)$ is a symplectic solvmanifold, since $\Gamma$ is a discrete subgroup, $\mathfrak{m}$ is the solvable Lie algebra $\mathfrak{r}$ of $\mathcal{R}$ and $\Omega$ induces a $\Gamma$-invariant symplectic form $\omega$ on it. Further, the product $\la$ defined by \eqref{na} for $a=0$ and $b=1$ is the canonical left symmetric product on $(\mathfrak{r},\omega)$ given by:	\begin{equation}\label{LiGrpLS}
\omega\left(\LS^{0,1}_x(y),z\right)=-\omega\left(y,[x,z]\right).\end{equation}
If moreover $\mathcal{R}$ is nilpotent, then $(\mathcal{R}/\Gamma,\Omega)$ is called a \textit{symplectic nilmanifold} and in this case $\mathcal{R}$ will be denoted by $N$, and $\mathfrak{r}$ by $\nlp$. Note that, since $\Omega$ is invariant, it follows from \cite[p. $340$]{Hu} that the torus is the only compact symplectic solvmanifold.
We end this section by providing an example of a symplectic solvmanifold.
\begin{example}
Consider the following simply connected Lie group
\begin{equation*}
\widetilde{\mathcal{R}}:=\Bigg\{\begin{psmallmatrix}
1&a&b\\0&1&c\\0&0&e^t\\
\end{psmallmatrix}\,\Big|\,\,(a,b,c,t)\in\rel^4\Bigg\}.
\end{equation*}
Let $\Gamma$ be the Lie subgroup of $\widetilde{\mathcal{R}}$ defined by:
\begin{equation*}
\Gamma:=\Bigg\{\begin{psmallmatrix}
1&n_1&n_2\\
0&1&n_3\\
0&0&1\\
\end{psmallmatrix}\,\Big|\,\,(n_1,n_2,n_3)\in\intg^3\Bigg\}.
\end{equation*}
Clearly that $\Gamma$ is a discrete Lie subgroup of $\widetilde{\mathcal{R}}$. Moreover, one can easily see that the Lie algebra $\mathfrak{r}$ of $\widetilde{\mathcal{R}}$ is:
\begin{equation*}
\mathfrak{r}=\Bigg\{\begin{psmallmatrix}
0&x&y\\0&0&z\\0&0&t\\
\end{psmallmatrix}\,\Big|\,\,(x,y,z,t)\in\rel^4\Bigg\}.
\end{equation*}
Choose the following basis of $\mathfrak{r}$:
\begin{equation*}
e_1:=\begin{psmallmatrix}
0&0&0\\0&0&-1\\0&0&0\\
\end{psmallmatrix},\quad e_2:=\begin{psmallmatrix}
0&1&0\\0&0&0\\0&0&0\\
\end{psmallmatrix},\quad e_3:=\begin{psmallmatrix}
0&0&1\\0&0&0\\0&0&0\\
\end{psmallmatrix},\et e_4:=\begin{psmallmatrix}
0&0&0\\0&0&0\\0&0&-1\\\end{psmallmatrix}.
\end{equation*}
The only nonzero Lie brackets are:
\begin{equation*}
[e_4,e_3]=[e_1,e_2]=e_3,\et [e_4,e_1]=e_1.
\end{equation*}
Hence, $\mathfrak{r}$ is $2$-step solvable. Moreover, define a symplectic tensor on $\mathfrak{r}$ by:
\begin{equation*}
\w:=e^*_1\wedge e^*_2-e^*_3\wedge e^*_4.
\end{equation*}
A direct computation shows that $\w$ is a $2$-cocycle of $\mathfrak{r}$. Further, for each $\gamma:=\begin{psmallmatrix}
1&n_1&n_2\\0&1&n_3\\0&0&1\\
\end{psmallmatrix}\in\Gamma$, we have:
\begin{eqnarray*}
&&\Ad_\gamma(e_1)=e_1-n_1e_3,\qquad\Ad_\gamma(e_2)=e_2-n_3e_3,\\
&&\Ad_\gamma(e_3)=e_3,\quad\et\,\,\Ad_\gamma(e_4)=e_4+n_3e_1-n_2e_3.
\end{eqnarray*}
It follows that $\w$ is $\Gamma$-invariant, and therefore it gives rise to a unique $\widetilde{\mathcal{R}}$-invariant symplectic form $\Omega$ on the solvmanifold $\widetilde{\mathcal{R}}/\Gamma\cong\mathbb{T}^3\times\rel$. In addition, the canonical left symmetric product on $\mathfrak{r}$ associated to $\omega$ defines the Zero-One connection $\nc$ on $\widetilde{\mathcal{R}}/\Gamma$.
\end{example}

\section{The flatness of the Zero-One connection $\nc$}
The purpose of this section is to investigate the conditions under which the Zero-One connection is flat, followed by a proof of our second main result.

We start by computing the curvature tensor of the Zero-One connection. Let $\ra:\mathfrak{m}\times\m\to\mathfrak{m}$ be the bilinear map defined by $\ra_x(y):=\la_y(x)$, and for each $x,y\in\mathfrak{m}$, denote by $D_{x,y}$ the following linear endomorphism
\begin{equation}\label{DxyDef}
D_{x,y}:\mathfrak{m}\to\mathfrak{m},\quad\textnormal{written}\quad z\mapsto[[x,z]_\h,y].
\end{equation}
Further, for any linear endomorphism $F\in\End(\mathfrak{m})$, denote by $F^\divideontimes\in\End(\mathfrak{m})$ the adjoint of $F$ with respect to $\omega$, i.e., $
\omega(F^\divideontimes(x),y)=\omega(x,F(y))$ for any $x,y\in\m$. To compute the curvature tensor of the Zero-One connection $\nc$, we need the following useful lemma. 

\begin{lemma}\label{LemLY3} For any $x,y,z,z'\in\mathfrak{m}$, we have:
\begin{itemize}
\item[$1.$] $D_{x,y}(z)=-D_{z,y}(x)$;
\item[$2.$] $\omega\left(D_{x,y}(z),z'\right)=\omega\left(D_{x,z'}(z),y\right)$;
\item[$3.$] $\displaystyle\sum_{\circlearrowright(x,y,z)} D_{x,z}(y)=\displaystyle\sum_{\circlearrowright(x,y,z)}\lm_x\ro\lm_y(z)$;
\item[$4.$] $\tr\left(D_{x,y}\right)-\tr\left(D_{y,x}\right)=\omega\left(u^\mathfrak{m},\lm_x(y)\right)$.
\end{itemize}
\end{lemma}

\begin{proof}
The first one is clear, the second one follows from the $\Ad(H)$-invariance of $\omega$, and the third one follows from Jacobi identity. The last one follows from the third property, i.e., for any $x,y\in\mathfrak{m}$, we have:
\begin{equation*}
D_{x,y}-D_{y,x}=\lm_{[x,y]_\m}+\left[\lm_y,\lm_x\right]+\ad_{[x,y]_\h},
\end{equation*}
and the fact that $\tr\left(\ad_{[x,y]_\h}\right)=0$.
\end{proof}

Now, we are able to give a formula for the curvature tensor of $\nc$.

\begin{proposition}
The curvature of the Zero-One connection $\nc$ is given by
\begin{equation}\label{Ka}
\Kc(x,y)=D_{x,y}^\aj-D_{y,x}^\aj\,.
\end{equation}
\end{proposition}

\begin{proof}
Recall that, for $x,y,z\in\mathfrak{m}$, the curvature of $\nc$ is given by $\eqref{Curv}$, i.e.,
\begin{equation*}
K^{0,1}(x,y)z=\left[\la_x,\la_y\right](z)-\la_{[x,y]_\mathfrak{m}}(z)-D_{x,z}(y).
\end{equation*}
So, for $z'\in\mathfrak{m}$, we have:
\begin{eqnarray*}
\omega\left(\left[\la_x,\la_y\right](z),z'\right)&=&-\omega\left(\la_y(z),[x,z']_\mathfrak{m}\right)+\omega\left(\la_x(z),[y,z']_\mathfrak{m}\right)\\
&=&\omega\left(z,[y,[x,z']_\mathfrak{m}]_\mathfrak{m}-[x,[y,z']_\mathfrak{m}]_\mathfrak{m}\right).
\end{eqnarray*}
On the other hand, by Lemma \ref{LemLY3}, we have:
\begin{eqnarray*}
\left[\lm_y,\lm_x\right](z')&=&-\lm_{[x,y]_\m}(z')-D_{x,z'}(y)+D_{x,y}(z')-D_{y,x}(z').
\end{eqnarray*}
Thus,
\begin{eqnarray*}
\omega\left(\left[\la_x,\la_y\right](z),z'\right)&=&-\omega\Big(z,\lm_{[x,y]_\m}(z')\Big)+
\omega\Big(D_{x,z'}(y),z\Big)\\
&&+\omega\Big(z,D_{x,y}(z')-D_{y,x}(z')\Big)\\
&=&\omega\Big(\la_{[x,y]_\mathfrak{m}}(z),z'\Big)+\omega\Big(D_{x,z}(y),z'\Big)\\
&&+\omega\Big(z,D_{x,y}(z')-D_{y,x}(z')\Big).
\end{eqnarray*}
This completes the proof.
\end{proof}

By the previous proposition, it is obvious that $\nc$ is flat if and only if $D_{x,y}=D_{y,x}$, for all $x,y\in\mathfrak{m}$. The following proposition gives a necessary and sufficient condition for the Zero-One connection $\nc$ to be flat.

\begin{proposition}\label{NaFlt}
The Zero-One connection $\nc$ is flat if and only if $$[[\mathfrak{m},\mathfrak{m}]_\h,\mathfrak{m}]=\{0\}.$$
In particular, if $\nc$ is flat, then $(\mathfrak{m},[\,.\,,.\,]_\mathfrak{m})$ is a Lie algebra. 
\end{proposition}

\begin{proof}
Only the "only if" part that needs to be checked. Suppose that $\nc$ is flat, then for each $x,y,z\in\mathfrak{m}$, one has $D_{x,y}(z)=D_{y,x}(z)$. Thus, by using Lemma \ref{LemLY3}, we get:
\begin{equation*}
D_{x,y}(z)=-D_{z,y}(x)=-D_{y,z}(x)=D_{x,z}(y).
\end{equation*}
Hence,
\begin{eqnarray*}
2D_{x,y}(z)&=&D_{x,y}(z)+D_{x,z}(y)\\
&=&D_{x,y}(z)-D_{y,z}(x)\\
&=&D_{x,y}(z)-D_{z,y}(x)\\
&=&D_{x,y}(z)-D_{z,x}(y)\\
&=&D_{x,y}(z)-D_{x,z}(y)\\
&=&D_{x,y}(z)-D_{x,y}(z)\\
&=&0.
\end{eqnarray*}
The last assertion follows from the third equality of Lemma \ref{LemLY3}.
\end{proof}

\begin{example}
On a symplectic solvmanifold $(\mathcal{R}/\Gamma,\Omega)$, the Zero-One connection $\nc$ is flat.
\end{example}

The natural action of $G$ on $G/H$ is said to be \textit{almost effective} if the largest normal subgroup of $G$ contained in $H$ is discrete. It is well known that the natural action of $G$ on $G/H$ is almost effective if and only if the isotropy representation $\ad^\m:\h\to\End(\mathfrak{m})$ is faithful. The following theorem is a direct consequence of Proposition \ref{NaFlt}. 

\begin{theorem}\label{ncFlt}
Let $(G/H,\Omega)$ be a symplectic reductive homogeneous $G$-space, $\g=\h\oplus\m$ a fixed reductive decomposition, and assume that $G$ acts almost effectively on $G/H$. Then, the Zero-One connection $\nc$ is flat if and only if $\mathfrak{m}$ is an ideal of $\mathfrak{g}$. 
\end{theorem}

Now, before proving our second main theorem, let us compute the Ricci curvature of the Zero-One connection. We start by the following useful lemma.

\begin{lemma}\label{tL,tR}
For any $x,y,z\in\mathfrak{m}$, the following equalities hold:
\begin{itemize}
\item[$1.$] $\omega\left(\ra_x(y),z\right)=\omega\left(y,\ra_x(z)\right)$;
\item[$2.$] $\tr(\la_x)=\tfrac{1}{2}\tr(\ra_x)=\omega\left(x,u^\mathfrak{m}\right)$;
\item[$3.$] $\tr(\la_x\ro\la_y)=\tr\left(\lm_x\ro\lm_y\right)$;
\item[$4.$] $\tr(\ra_x\ro\ra_y)=2\tr(\ra_x\ro\la_y)=2\tr\left(\lm_x\ro\lm_y\right)+2\tr\left(\lm_x\ro\left(\lm_y\right)^\aj\right)$.
\end{itemize}
\end{lemma}

\begin{proof}
A straightforward computation, using
\begin{equation*}
\la_x=-\left(\lm_x\right)^\divideontimes,\qquad\ra_x=-\left(\lm_x+\left(\lm_x\right)^\divideontimes\right),
\end{equation*}
for $x\in\m$, and the fact that $\tr\left(F\right)=\tr\left(F^\aj\right)$, for any $F\in\End(\mathfrak{m})$.
\end{proof}

The \textit{Ricci curvature} of $\nc$ is defined by:
\begin{equation*}
\ric^{0,1}(x,y):=\tr\big(z\mapsto \Kc(x,z)y\big),\qquad\forall\,x,y\in\m.
\end{equation*}

\begin{proposition}
The Ricci curvature of the Zero-One connection $\nc$ is given by
\begin{equation}\label{RicA}
\ric^{0,1}(x,y)=2\tr\left(\lm_x\ro\lm_y\right)+2\tr\left(\lm_x\ro\left(\lm_y\right)^\aj\right)+2\,\omega\left(\lm_{u^\mathfrak{m}}(x),y\right)-\tr(D_{x,y}).
\end{equation}
\end{proposition}

\begin{proof} 
From $\eqref{Curv}$ and the fact that $\nc$ is torsion-free, we have:
\begin{equation}\label{CalRicA}
\Kc(x,\,\cdot\,)y=\left[\la_x,\ra_y\right]+\ra_y\ro\ra_x-\ra_{\la_x(y)}-D_{x,y}.
\end{equation}
Therefore,
\begin{equation*}
\ric^{0,1}(x,y)=\tr(\ra_x\ro\ra_y)-\tr(\ra_{\la_x(y)})-\tr(D_{x,y}).
\end{equation*}
Thus, the result follows by using Lemma \ref{tL,tR}.
\end{proof}

It is well known (see \cite[p. $14$]{Nom2}) that a torsion-free connection on a smooth manifold $M$ has symmetric Ricci curvature if and only if around each point of $M$ there is a local parallel volume form of $M$. In general, the Ricci curvature of the Zero-One connection $\nc$ is not symmetric. However, we have:

\begin{corollary}\label{SymRicA}
The Ricci curvature of $\nc$ is symmetric if and only if
\begin{equation}
\tr\left(\ad_{[x,y]_\mathfrak{m}}\right)=0,
\end{equation}
for any $x,y\in\mathfrak{m}$. This is in particular the case when $\mathfrak{g}$ is unimodular.
\end{corollary}
\begin{proof}
For $x,y\in\mathfrak{m}$, we have:
\begin{eqnarray*}
\ric^{0,1}(x,y)-\ric^{0,1}(y,x)&=&2\,\omega\left(\lm_{u^\mathfrak{m}}(x),y\right)-2\,\omega\left(\lm_{u^\mathfrak{m}}(y),x\right)-\big\{\tr(D_{x,y})-\tr(D_{y,x})\big\}\\
&\stackrel{\rm\ref{LemLY3}}{=}&2\,\omega\left(\lm_{u^\mathfrak{m}}(x),y\right)+2\,\omega\left(x,\lm_{u^\mathfrak{m}}(y)\right)-\omega\left(u^\mathfrak{m},\lm_x(y)\right)\\
&\stackrel{\rm\eqref{Cocy}}{=}&\omega\left(u^\mathfrak{m},\lm_x(y)\right)\\
&=&\tr\left(\lm_{[x,y]_\m}\right).
\end{eqnarray*}
On the other hand, let $\{f_j\}_{1\leqslant j\leqslant \ell}$ be a basis of $\h$, and let $\{e_j\}_{1\leqslant j\leqslant k}$ be a basis of $\m$. For any $z\in\m$, we have:
\begin{eqnarray*}
	\tr(\ad_z)=\somjl f^*_j\big([z,f_j]\big)+\somjk e_j^*\big([z,e_j]\big)=\somjk e^*_j\big([z,e_j]_\m\big)=\tr\left(\lm_z\right).
\end{eqnarray*}
This gives that $\tr\left(\lm_z\right)=\tr\left(\ad_z\right)$, for any $z\in\mathfrak{m}$, proving the claim.
\end{proof}

Now we will prove our second main result. First, since $G/H$ is connected, it is well known that the identity component $G^0$ of $G$ acts transitively on $G/H$, and therefore, throughout the proof, we may assume without loss of generality that $G$ is connected.

\begin{proof}[\textnormal{\textbf{Proof of Theorem \ref{Theo1Flat1}}}] For the first statement, since $\nc$ is flat, it follows from Theorem \ref{ncFlt} that $\m$ is an ideal of $\g$. Consequently, the integral subgroup $G_\m:=\langle\exp_G(\m)\rangle$ corresponding to the Lie subalgebra $\m\subset\g$ is a connected normal (immersed) Lie subgroup of $G$. Moreover, the pair $(\mathfrak{m}, \omega)$ forms a symplectic Lie algebra. Regarding the second assertion, by using the fact that $[\h,\m]\subseteq\m$, it follows from \cite[p. $439$]{Nb} that $G=G_\m H$. Hence, the restriction of the smooth action $G\times G/H\to G/H$ to $G_\m$ is transitive. Furthermore, since $G_\m\subset G$, it is evident that $G_\m$ preserves both the symplectic form $\Omega$ and the Zero-One connection $\nc$. On the other hand, a small computation shows that the isotropy group at $H\in G/H$ is equal to the closed Lie subgroup $G_\m\cap H$ of $G_\m$. Moreover, since the Lie algebra of $G_\m\cap H$ is $\m\cap\h=\{0\}$, it follows that $G_\m\cap H$ is discrete. For the third statement, using that $\Gamma\subset G_\m$ is a discrete Lie subgroup of $G_\m$, we obtain that $\pi_{G_\m}:G_\m\to G_\m/\Gamma$ is a smooth covering map. But since $G_\m/\Gamma=G/H$ is simply connected, one can conclude that $\pi_{G_\m}:G_\m\to G_\m/\Gamma$ is a diffeomorphism. For the last assertion, when $G/H$ is compact, it can be deduced from \cite[p. $340$]{Hu} that $G_\mathfrak{m}$ is abelian, and this finishes the proof.
\end{proof}

Note that, if we assume that $G$ is simply connected, as we can always do by replacing $G$ with its universal covering group $\widetilde{G}$, we can establish that $G_\m$ is closed in $G$ (cf. \cite[p. $440$]{Nb}).\\

We end this section by a direct consequence of Theorem \ref{Theo1Flat1}.
\begin{corollary}
Let $(G/H,\Omega)$ be a compact, simply connected, symplectic reductive homogeneous $G$-space, and assume that $G$ acts almost effectively on $G/H$. Then, the Zero-One connection $\nc$ is never flat.
\end{corollary}
\section{The natural symplectic connection $\ns$}
In this section,we will compute the curvature and the Ricci curvature tensors of the natural symplectic connection $\ns:=\n^{\frac{1}{3},\frac{1}{3}}$, prove our third main result, and then present some explicit examples. In what follows, we denote $\ls:=\LS^{\frac{1}{3},\frac{1}{3}}$, and define the bilinear map $\rs:\mathfrak{m}\times\m\to\mathfrak{m}$ by $\rs_x(y):=\ls_y(x)$. The following proposition is a straightforward computation.
\begin{proposition}\label{L,R}
For any $x,y\in\mathfrak{m}$, the following formulas hold:
\begin{itemize}
\item[$1.$] $\ls_x=\frac{2}{3}\la_x-\frac{1}{3}\ra_x=\frac{1}{3}\left(\lm_x-\left(\lm_x\right)^\aj\right)$;
\item[$2.$] $\rs_x=\frac{2}{3}\ra_x-\frac{1}{3}\la_x=-\frac{1}{3}\left(2\lm_x+\left(\lm_x\right)^\aj\right)$;
\item[$3.$] $\omega\left(\ls_x(y),z\right)=-\omega\left(y,\ls_x(z)\right)$;
\item[$4.$] $\tr(\ls_x)=0$, and $\,\tr(\rs_x)=\tr(\la_x)$.
\end{itemize}
\end{proposition}

The next proposition gives a formula for the curvature tensor of $\ns$. 
\begin{proposition}
The curvature of the natural symplectic connection $\ns$ is given by
\begin{equation}\label{Ks}
\Ks(x,y)=\tfrac{1}{9}\Big\{\ra_{[x,y]_\mathfrak{m}}-\left[\ra_x,\ra_y\right]-2\la_{[x,y]_\mathfrak{m}}\Big\}+\tfrac{2}{9}\Big\{\Kc(x,y)-\Kc(x,y)^\aj-\tfrac{5}{2}\ad_{[x,y]_\h}\Big\},
\end{equation}
where $\Kc$ is the curvature of the Zero-One connection $\nc$ given by \eqref{Ka}. 
\end{proposition}
\begin{proof}
For $x,y\in\mathfrak{m}$, the curvature of $\ns$ is given by:
\begin{equation*}
\Ks(x,y)=\left[\ls_x,\ls_y\right]-\ls_{[x,y]_\mathfrak{m}}-\ad_{[x,y]_\h}.
\end{equation*}
Using Proposition \ref{L,R}, we get:
\begin{eqnarray*}
\Ks(x,y)&=&\tfrac{4}{9}\left[\la_x,\la_y\right]-\tfrac{2}{9}\left[\la_x,\ra_y\right]+\tfrac{2}{9}\left[\la_y,\ra_x\right]\\
&&+\tfrac{1}{9}\left[\ra_x,\ra_y\right]-\tfrac{2}{3}\la_{[x,y]_\mathfrak{m}}+\tfrac{1}{3}\ra_{[x,y]_\mathfrak{m}}-\ad_{[x,y]_\h}\\
&=&\tfrac{1}{9}\Big\{4\Kc(x,y)-2\la_{[x,y]_\mathfrak{m}}-5\ad_{[x,y]_\h}+\left[\ra_x,\ra_y\right]\\
&&+3\ra_{[x,y]_\mathfrak{m}}-2\left[\la_x,\ra_y\right]+2\left[\la_y,\ra_x\right]\Big\}.
\end{eqnarray*}
On the other hand, using \eqref{CalRicA}, we have: 
\begin{eqnarray*}
\left[\la_y,\ra_x\right](z)-\left[\la_x,\ra_y\right](z)&=&\Kc(y,z)x-\ra_x\ro\ra_y(z)+\ra_{\la_y(x)}(z)+D_{y,x}(z)\\
&&-\Kc(x,z)y+\ra_y\ro\ra_x(z)-\ra_{\la_x(y)}(z)-D_{x,y}(z)\\
&=&-\left[\ra_x,\ra_y\right](z)-\ra_{[x,y]_\mathfrak{m}}(z)-\left(D_{x,y}-D_{y,x}\right)(z)\\
&&+\Kc(y,z)x+\Kc(z,x)y\\
&=&-\left[\ra_x,\ra_y\right](z)-\ra_{[x,y]_\mathfrak{m}}(z)-\Kc(x,y)^\aj z-\Kc(x,y)z,
\end{eqnarray*}
for $z\in\mathfrak{m}$, where the second equality follows from the fact that the Zero-One connection $\nc$ is torsion-free, and the third one form Bianchi's identity. 
\end{proof}

Similarly to the Zero-One connection, the \textit{Ricci curvature} of the natural symplectic connection $\ns$ is defined by:
\begin{equation*}
\ric^\mathbf{s}(x,y):=\tr\big(z\mapsto \Ks(x,z)y\big).
\end{equation*}
In this case, since $\ns$ is symplectic, the Ricci curvature $\ric^\mathbf{s}$ is symmetric, i.e., $\ric^\mathbf{s}(x,y)=\ric^\mathbf{s}(y,x)$ for any $x,y\in\mathfrak{m}$. The following proposition gives a formula for $\ric^\mathbf{s}$.

\begin{proposition}
The Ricci curvature of the natural symplectic connection $\ns$ is given by
\begin{equation}\label{RicS}
\ric^\mathbf{s}(x,y)=\tfrac{1}{9}\Big\{\tr\left(\lm_x\ro\lm_y\right)+\omega\left(x,\lm_{u^\mathfrak{m}}(y)\right)\Big\}+\tfrac{2}{9}\Big\{\ric^{0,1}(x,y)-\tr(D_{y,x})-\tfrac{5}{2}\tr(D_{x,y})\Big\},
\end{equation}
where $\ric^{0,1}$ is the Ricci curvature of $\nc$ given by \eqref{RicA}, and $D_{x,y}$ is the linear endomorphism defined in \eqref{DxyDef}.
\end{proposition}

\begin{proof}
For any $x,y,z\in\mathfrak{m}$, a direct computation gives: 
\begin{eqnarray*}
\tr\left(z\mapsto\ra_{[x,z]_\mathfrak{m}}(y)\right)&=&\tr(\la_y\ro\la_x)-\tr(\la_y\ro\ra_x);\\
\tr\big(z\mapsto\left[\ra_x,\ra_z\right](y)\big)&=&\tr(\ra_x\ro\la_y)-\tr(\la_{\ra_x(y)});\\
\tr\left(z\mapsto\la_{[x,z]_\mathfrak{m}}(y)\right)&=&\tr(\ra_y\ro\la_x)-\tr(\ra_y\ro\ra_x);\\
\tr\left(z\mapsto \Kc(x,z)^\aj y\right)&=&\tr(D_{y,x}).
\end{eqnarray*}
Thus, using \eqref{Ks} and Lemma \ref{tL,tR}, the equality \eqref{RicS} follows.
\end{proof}

For a symplectic Lie group $(G,\Omega)$, we have $\lm=\ad$, and $D_{x,y}=0$, for any $x,y\in\mathfrak{g}$. Further, we denote by $u^\mathfrak{g}\in\mathfrak{g}$ the unique vector of $\mathfrak{g}$ which satisfies $\omega\left(u^\mathfrak{g},x\right)=\tr\left(\ad_x\right)$, for all $x\in\mathfrak{g}$. The following proposition, which follows directly from \eqref{Ks} and \eqref{RicS}, gives the curvature and the Ricci curvature of the natural symplectic connection $\ns$ of a symplectic Lie group.
\begin{proposition}
On a symplectic Lie group $(G,\Omega)$, the natural symplectic connection $\ns$ is defined by the following product
\begin{equation}\label{PrdLiGrp}
\ls_x(y)=\tfrac{1}{3}\Big\{\la_x(y)+[x,y]\Big\},
\end{equation}
where $\la:\mathfrak{g}\times\mathfrak{g}\to\mathfrak{g}$ is the canonical left symmetric product on $(\mathfrak{g},\omega)$. Moreover, the curvature of $\ns$ is given by
\begin{equation}\label{KsLi}
\Ks(x,y)=\tfrac{1}{9}\left\{\ra_{[x,y]}-\left[\ra_x,\ra_y\right]-2\la_{[x,y]}\right\}.
\end{equation}
Further, the Ricci curvature of $\ns$ is given by
\begin{equation}\label{RisLi}
\ric^\mathbf{s}(x,y)=\tfrac{1}{9}\Big\{\kappa_{\mathfrak{g}}\left(x,y\right)+\omega\left(x,\ad_{u^{\mathfrak{g}}}(y)\right)\Big\},
\end{equation}
for $x,y\in\mathfrak{g}$, where $\kappa_\mathfrak{g}:\mathfrak{g}\times\mathfrak{g}\to{\mathbb{R}}$ is the Killing form of $\mathfrak{g}$.
\end{proposition}

The next corollary gives an algebraic characterization of the Ricci-flatness of the natural symplectic connection on a symplectic unimodular Lie group.

\begin{corollary}\label{UniLiGrp}
Let $(G,\Omega)$ be a symplectic unimodular Lie group. The natural symplectic connection $\ns$ is Ricci-flat if and only if the Killing form of $\mathfrak{g}$ vanishes.
\end{corollary}

\begin{example}
Consider the Lie algebra $\mathfrak{r}_6$ which is defined by the following Lie brackets:
\begin{equation*}
	[e_1,e_5]=e_2,\qquad [e_2,e_5]=e_1,\qquad [e_3,e_5]=e_4,\et [e_4,e_5]=-e_3,
\end{equation*}
where $\mathbb{B}:=\{e_j\}_{1\leqslant j\leqslant 6}$ is a basis of $\mathfrak{r}_6$, and the other brackets are either zero or given by skew-symmetry. It is clear that $\mathfrak{r}_6$ is a $2$-step solvable, non-nilpotent Lie algebra. Moreover, we have:
\begin{eqnarray*}
	&&\mathcal{M}_\mathbb{B}\big(\ad_{e_1}\big)=\begin{psmallmatrix}
		0 & 0 & 0 & 0 & 0 & 0 
		\\
		0 & 0 & 0 & 0 & 1 & 0 
		\\
		0 & 0 & 0 & 0 & 0 & 0 
		\\
		0 & 0 & 0 & 0 & 0 & 0 
		\\
		0 & 0 & 0 & 0 & 0 & 0 
		\\
		0 & 0 & 0 & 0 & 0 & 0 
	\end{psmallmatrix},\quad\quad
	\mathcal{M}_\mathbb{B}\big(\ad_{e_2}\big)=\begin{psmallmatrix}
		0 & 0 & 0 & 0 & 1 & 0 
		\\
		0 & 0 & 0 & 0 & 0 & 0 
		\\
		0 & 0 & 0 & 0 & 0 & 0 
		\\
		0 & 0 & 0 & 0 & 0 & 0 
		\\
		0 & 0 & 0 & 0 & 0 & 0 
		\\
		0 & 0 & 0 & 0 & 0 & 0 
	\end{psmallmatrix},\\ &&\mathcal{M}_\mathbb{B}\big(\ad_{e_3}\big)=\begin{psmallmatrix}
		0 & 0 & 0 & 0 & 0 & 0 
		\\
		0 & 0 & 0 & 0 & 0 & 0 
		\\
		0 & 0 & 0 & 0 & 0 & 0 
		\\
		0 & 0 & 0 & 0 & 1 & 0 
		\\
		0 & 0 & 0 & 0 & 0 & 0 
		\\
		0 & 0 & 0 & 0 & 0 & 0 
	\end{psmallmatrix},\quad\quad
	\mathcal{M}_\mathbb{B}\big(\ad_{e_4}\big)=\begin{psmallmatrix}
		0 & 0 & 0 & 0 & 0 & 0 
		\\
		0 & 0 & 0 & 0 & 0 & 0 
		\\
		0 & 0 & 0 & 0 & -1 & 0 
		\\
		0 & 0 & 0 & 0 & 0 & 0 
		\\
		0 & 0 & 0 & 0 & 0 & 0 
		\\
		0 & 0 & 0 & 0 & 0 & 0 
	\end{psmallmatrix},\\
	&&\mathcal{M}_\mathbb{B}(\ad_{e_5})=\begin{psmallmatrix}
		0 & -1 & 0 & 0 & 0 & 0 
		\\
		-1 & 0 & 0 & 0 & 0 & 0 
		\\
		0 & 0 & 0 & 1 & 0 & 0 
		\\
		0 & 0 & -1 & 0 & 0 & 0 
		\\
		0 & 0 & 0 & 0 & 0 & 0 
		\\
		0 & 0 & 0 & 0 & 0 & 0 
	\end{psmallmatrix},\et
	\mathcal{M}_\mathbb{B}(\ad_{e_6})=\begin{psmallmatrix}
		0 & 0 & 0 & 0 & 0 & 0 
		\\
		0 & 0 & 0 & 0 & 0 & 0 
		\\
		0 & 0 & 0 & 0 & 0 & 0 
		\\
		0 & 0 & 0 & 0 & 0 & 0 
		\\
		0 & 0 & 0 & 0 & 0 & 0 
		\\
		0 & 0 & 0 & 0 & 0 & 0 
	\end{psmallmatrix}.
\end{eqnarray*}
Thus, $\mathfrak{r}_6$ is unimodular. On the other hand, one can easily check that $\w:=\sum_{j=0}^2e^*_{2j+1}\wedge e^*_{2(j+1)}$, is a symplectic form on $\mathfrak{r}_6$. Further, a straightforward computation yields that:
\begin{equation*}
	\kappa^{\mathfrak{r}_6}(x,y)=0,\qquad\forall\,x,y\in\mathfrak{r}_6.
\end{equation*}
As a result, using Corollary \ref{UniLiGrp}, we get that the simply connected symplectic Lie group $(\widetilde{\mathcal{R}}_6,\Omega)$ corresponding to $(\mathfrak{r}_6,\w)$ is Ricci-flat.
\end{example}

Clearly that for a symplectic solvmanifold $(\mathcal{R}/\Gamma,\Omega)$, the natural symplectic connection $\ns$ is associated to the product \eqref{PrdLiGrp}, where $\la:\mathfrak{r}\times\mathfrak{r}\to\mathfrak{r}$ is the canonical left symmetric product on $(\mathfrak{r},\omega)$. Moreover, the curvature and the Ricci curvature tensors of the natural symplectic connection $\ns$ are given respectively by \eqref{KsLi} and \eqref{RisLi}. The following proposition is a direct consequence of Corollary \ref{UniLiGrp}.

\begin{proposition}\label{NilRicFlt}
The natural symplectic connection $\ns$ of a symplectic nilmanifold $(N/\Gamma,\Omega)$ is Ricci-flat.
\end{proposition}

This section concludes by proving our third main result. Following this, we present some explicit examples.

\begin{proof}[\textnormal{\textbf{Proof of Theorem \ref{ThmSympNilman}}}]
According to Theorem \ref{Theo1Flat1}, we have $(G/H,\Omega)=(G_\m/\Gamma,\Omega)$ with $G_\m$ is a connected Lie group and $\Gamma$ a discrete Lie subgroup of $G_\m$. On the other hand, since $\n^\mathbf{s}$ is flat, it follows that the Lie algebra $(\m,\w)$ of $G_\m$ endowed with $\w$ is a flat symplectic Lie algebra (cf. \cite{Oili}). Therefore, by using \cite[p. $4390$]{Oili}, we deduce that $\m$ must be a nilpotent Lie algebra, and hence the statement is proven.
\end{proof}

\begin{example}\label{rriF1}
Consider the following simply connected Lie group
\begin{equation*}
\widetilde{N}:=\Bigg\{\begin{psmallmatrix}
e^t&0&0&0\\
0&1&a&b\\
0&0&1&c\\
0&0&0&1\\
\end{psmallmatrix}\,\Big|\,\,(a,b,c,t)\in\rel^4\Bigg\}.
\end{equation*}
Let $\Gamma$ be the Lie subgroup of $\widetilde{N}$ defined by: 
\begin{equation*}
\Gamma:=\Bigg\{\begin{psmallmatrix}
e^{n_1}&0&0&0\\
0&1&0&n_2\\
0&0&1&n_3\\
0&0&0&1\\
\end{psmallmatrix}\,\Big|\,\,(n_1,n_2,n_3)\in\intg^3\Bigg\}.
\end{equation*}
Clearly that $\Gamma$ is a discrete Lie subgroup of $\widetilde{N}$. Moreover, one can easily see that the Lie algebra $\nl$ of $\widetilde{N}$ is:
\begin{equation*}
\nl=\Bigg\{\begin{psmallmatrix}
t&0&0&0\\
0&0&x&y\\
0&0&0&z\\
0&0&0&0\\
\end{psmallmatrix}\,\Big|\,\,(x,y,z,t)\in\rel^4\Bigg\}.
\end{equation*}
Choose the following basis of $\nl$:
\begin{equation*}
e_1:=\begin{psmallmatrix}
1&0&0&0\\
0&0&0&0\\
0&0&0&0\\
0&0&0&0\\
\end{psmallmatrix},\quad e_2:=\begin{psmallmatrix}
0&0&0&0\\
0&0&1&0\\
0&0&0&0\\
0&0&0&0\\
\end{psmallmatrix},\quad e_3:=\begin{psmallmatrix}
0&0&0&0\\
0&0&0&1\\
0&0&0&0\\
0&0&0&0\\
\end{psmallmatrix},\et e_4:=\begin{psmallmatrix}
0&0&0&0\\
0&0&0&0\\
0&0&0&1\\
0&0&0&0\\
\end{psmallmatrix}.
\end{equation*}
Since $[e_2,e_4]=e_3$ is the only nonzero Lie bracket, we can obviously see that $\nl$ is $2$-step nilpotent and $\mathfrak{z}(\nl)=\Span_\rel\{e_1,e_3\}$. Moreover, define a symplectic tensor on $\nl$ by:
\begin{equation*}
\w:=e^*_1\wedge e^*_4+e^*_2\wedge e^*_3.
\end{equation*}
A direct computation shows that $\w$ is a $2$-cocycle of $\nl$. Further, for each $\gamma:=\begin{psmallmatrix}
e^{n_1}&0&0&0\\
0&1&0&n_2\\
0&0&1&n_3\\
0&0&0&1\\
\end{psmallmatrix}\in\Gamma$, we have:
\begin{equation*}
\Ad_\gamma(e_1)=e_1,\quad\Ad_\gamma(e_2)=e_2-n_3e_3,\quad\Ad_\gamma(e_3)=e_3,\et\Ad_\gamma(e_4)=e_4.
\end{equation*}
Thus, $\w$ is $\Gamma$-invariant, and therefore it gives rise to a unique $\widetilde{N}$-invariant symplectic form $\Omega$ on the nilmanifold $\widetilde{N}/\Gamma\cong\mathbb{T}^3\times\rel$. Furthermore, by a straightforward verification one can check that $\mathfrak{z}(\nl)$ is Lagrangian in $(\nl,\w)$. In particular, the natural symplectic connection $\n^\mathbf{s}$ is flat (cf. \cite[p. $4384$]{Oili}).
\end{example}

As indicated in Proposition \ref{NilRicFlt}, the natural symplectic connection $\ns$ of any symplectic nilmanifold is Ricci-flat. Nevertheless, it is essential to note that the natural symplectic connection is not necessarily flat, as demonstrated by the following example.

\begin{example}\label{rriF2}
Consider the following simply connected Lie group
\begin{equation*}
\widetilde{N}:=\Bigg\{\begin{psmallmatrix}
1&t&\tfrac{t^2}{2}&a\\
0&1&t&b\\
0&0&1&c\\
0&0&0&1\\
\end{psmallmatrix}\,\Big|\,\,(a,b,c,t)\in\rel^4\Bigg\}.
\end{equation*}
Let $\Gamma$ be the discrete Lie subgroup of $\widetilde{N}$ defined by: 
\begin{equation*}
\Gamma:=\Bigg\{\begin{psmallmatrix}
1&0&0&n_1\\
0&1&0&n_2\\
0&0&1&0\\
0&0&0&1\\
\end{psmallmatrix}\,\Big|\,\,(n_1,n_2)\in\intg^2\Bigg\}.
\end{equation*}
Clearly that the Lie algebra $\nl$ of $\widetilde{N}$ is:
\begin{equation*}
\nl=\Bigg\{\begin{psmallmatrix}
0&t&0&x\\
0&0&t&y\\
0&0&0&z\\
0&0&0&0\\
\end{psmallmatrix}\,\Big|\,\,(x,y,z,t)\in\rel^4\Bigg\}.
\end{equation*}
Choose the following basis of $\nl$:
\begin{equation*}
e_1:=\begin{psmallmatrix}
0&0&0&-1\\
0&0&0&0\\
0&0&0&0\\
0&0&0&0\\
\end{psmallmatrix},\quad e_2:=\begin{psmallmatrix}
0&0&0&0\\
0&0&0&1\\
0&0&0&0\\
0&0&0&0\\
\end{psmallmatrix},\quad e_3:=\begin{psmallmatrix}
0&0&0&0\\
0&0&0&0\\
0&0&0&-1\\
0&0&0&0\\
\end{psmallmatrix},\et e_4:=\begin{psmallmatrix}
0&1&0&0\\
0&0&1&0\\
0&0&0&0\\
0&0&0&0\\
\end{psmallmatrix}.
\end{equation*}
The only nonzero Lie brackets are $[e_2,e_4]=e_1$, and $[e_3,e_4]=e_2$. In particular, $\nl$ is a $3$-step nilpotent Lie algebra. Define a symplectic tensor on $\nl$ by:
\begin{equation*}
\w:=e^*_1\wedge e^*_4+e^*_2\wedge e^*_3.
\end{equation*}
A direct computation shows that $\w$ is a $2$-cocycle of $\nl$. Moreover, for each $\gamma:=\begin{psmallmatrix}
1&0&0&n_1\\
0&1&0&n_2\\
0&0&1&0\\
0&0&0&1\\
\end{psmallmatrix}\in\Gamma$, we have:
\begin{equation*}
\Ad_\gamma(e_1)=e_1,\quad\Ad_\gamma(e_2)=e_2,\quad\Ad_\gamma(e_3)=e_3,\et\Ad_\gamma(e_4)=e_4+n_2e_1.
\end{equation*}
Hence, $\w$ is $\Gamma$-invariant, and therefore it gives rise to a unique $\widetilde{N}$-invariant symplectic form $\Omega$ on the nilmanifold $\widetilde{N}/\Gamma\cong\mathbb{T}^2\times\rel^2$. On the other hand, since $\{e_i\}_{1\leqslant i\leqslant 4}$ is a symplectic basis of $(\nl,\w)$, one can easily check that 
\begin{equation*}
\la_{e_2}(e_4)=\la_{e_3}(e_3)=e_1,\quad \la_{e_4}(e_3)=-e_2,\et\la_{e_4}(e_4)=e_3,
\end{equation*}
are the only nonzero elements of the product $\la$. Thus, we compute:  
\begin{eqnarray*}
K^\mathbf{s}(e_3,e_4)e_4&=&\tfrac{1}{9}\ra_{[e_3,e_4]}(e_4)-\tfrac{1}{9}\left[\ra_{e_3},\ra_{e_4}\right](e_4)-\tfrac{2}{9}\la_{[e_3,e_4]}(e_4)\\
&=&\tfrac{1}{9}\ra_{e_2}(e_4)-\tfrac{1}{9}\ra_{e_3}\ro\ra_{e_4}(e_4)+\tfrac{1}{9}\ra_{e_4}\ro\ra_{e_3}(e_4)-\tfrac{2}{9}\la_{e_2}(e_4)\\
&=&\tfrac{1}{9}\la_{e_4}(e_2)-\tfrac{1}{9}\ra_{e_3}\ro\la_{e_4}(e_4)+\tfrac{1}{9}\ra_{e_4}\ro\la_{e_4}(e_3)-\tfrac{2}{9}e_1\\
&=&-\tfrac{1}{9}\ra_{e_3}(e_3)-\tfrac{1}{9}\ra_{e_4}(e_2)-\tfrac{2}{9}e_1\\
&=&-\tfrac{1}{9}\la_{e_3}(e_3)-\tfrac{1}{9}\la_{e_2}(e_4)-\tfrac{2}{9}e_1\\
&=&-\tfrac{4}{9}e_1.
\end{eqnarray*}
This shows that $\n^\mathbf{s}$ is not flat.
\end{example}

\section{Symplectic homogeneous spaces with compact Lie group}\label{Section4}

Throughout this section, we assume that $G$ is compact. Let $\met$ be an $\Ad(G)$-invariant inner product on $\mathfrak{g}$, and $\mathfrak{g}=\mathfrak{h}\oplus\mathfrak{m}$ the natural reductive decomposition of $\mathfrak{g}$ with respect to $\met$, i.e., $\mathfrak{m}:=\mathfrak{h}^\perp$.
Using the fact that $\met$ is $\Ad(G)$-invariant and $\mathfrak{m}$ is the orthogonal complement of $\h$ with respect to $\met$, one can easily see that for any $x,y,z\in\mathfrak{m}$, we have:
\begin{equation}\label{admSkew}
\langle \lm_x(y),z\rangle =-\langle y,\lm_x(z)\rangle .
\end{equation}
In particular, $u^\mathfrak{m}=0$. Define a linear isomorphism $\F:\mathfrak{m}\to\mathfrak{m}$ of $\mathfrak{m}$ by:
\begin{equation}\label{W}
\omega(x,y)=\langle \F(x),y\rangle ,
\end{equation}
for $x,y\in\mathfrak{m}$. The following proposition summarize all the properties of $\F$.
\begin{proposition}\label{Cpt}
For any $x,y,z\in\mathfrak{m}$, we have:
\begin{itemize}
\item[$1.$] $\langle \F(x),y\rangle =-\langle  x,\F(y)\rangle $;
\item[$2.$] $\F\ro\Ad_h=\Ad_h\ro\F$, for all $h\in H$;
\item[$3.$] $\big[\F,\lm_x\big]=\lm_{\F(x)}$;
\item[$4.$] $\la_x=\F^{-1}\ro\lm_x\ro\F$;
\item[$5.$] $D^\aj_{x,y}=\F^{-1}\ro\,D_{y,x}\ro\F$.
\end{itemize}
\end{proposition}

\begin{proof}
The first and the second one are obvious. The third equality follows from \eqref{Cocy} and \eqref{admSkew}. The fourth property also follows from \eqref{admSkew}. For the last one, a small computation shows that for any $x,y,z,z'\in\mathfrak{m}$, we have:
\begin{equation*}
\langle D_{x,y}(z),z'\rangle =\langle z,D_{y,x}(z')\rangle .
\end{equation*}
Hence, the last equality follows.
\end{proof}

In this case, the formula \eqref{Ka} of the curvature of the Zero-One connection $\nc$ became 
\begin{equation*}
\Kc(x,y)=\F^{-1}\ro\left(D_{y,x}-\,D_{x,y}\right)\ro\F,\qquad\forall\,x,y\in\mathfrak{m}.
\end{equation*}
Further, we have the following:

\begin{proposition}
The Ricci curvature of $\nc$ is symmetric and it is given by
\begin{equation}\label{RicACpt}
\ric^{0,1}(x,y)=2\tr\left(\lm_x\ro\F^{-1}\ro\lm_{\F(y)}\right)-\tr(D_{x,y}).
\end{equation}
\end{proposition}

\begin{proof}
The fact that $\nc$ has symmetric Ricci curvature follows from Corollary \ref{SymRicA}, since $\mathfrak{g}$ is unimodular. Moreover, by Proposition \ref{Cpt}, we have:
\begin{eqnarray*}
\lm_x\ro\left(\lm_y\right)^\aj&=&-\lm_x\ro\F^{-1}\ro\lm_y\ro\F\\
&=&\lm_x\ro\F^{-1}\ro\lm_{\F(y)}-\lm_x\ro\lm_y.
\end{eqnarray*}
Hence, the equality \eqref{RicACpt} follows by substituting in \eqref{RicA}.
\end{proof}

Combining \eqref{RicS} with \eqref{RicACpt} we get:
\begin{proposition}
The Ricci curvature of the natural symplectic connection $\ns$ is given by
\begin{equation}\label{RicSCpt}
\ric^\mathbf{s}(x,y)=\tfrac{1}{9}\tr\left(\lm_x\ro\lm_y\right)+\tfrac{4}{9}\tr\left(\lm_x\ro\F^{-1}\ro\lm_{\F(y)}\right)-\tr\left(D_{x,y}\right).
\end{equation}
\end{proposition}

Note that, the second property in Proposition \ref{Cpt} shows that the linear endomorphism $\F\in\End(\mathfrak{m})$ gives rise to a $G$-invariant $(1,1)$-tensor field (also denoted by $\F$) on $G/H$. We end this section by showing that this tensor field is not parallel with respect to the natural symplectic connection.

\begin{proposition}
The $(1,1)$-tensor field $\F$ on $G/H$ is not parallel with respect to the natural symplectic connection $\ns$ of $G/H$.
\end{proposition}

\begin{proof}
Let $x\in\mathfrak{m}$, and $x^\#\in\vecf(G/H)$ its corresponding fundamental vector field. Since under the identification of $T_H(G/H)$ with $\mathfrak{m}$, we have:
\begin{eqnarray*}
\ns_{x^\#_H}\F&=&\big[\ls_x,\F\big].
\end{eqnarray*}
It follows that $\F$ is parallel with respect to $\ns$ if and only if $\left[\ls_x,\F\right]=0$, which is equivalent (by using Propositions \ref{L,R} and \ref{Cpt}) to
\begin{equation}\label{Wprl}
\lm_x\ro\F=-\F\ro\lm_x.
\end{equation}
We argue by contradiction. Assume that \eqref{Wprl} holds, then for any $x,y\in\mathfrak{m}$, we have:
\begin{eqnarray*}
3\F\big([x,y]_\mathfrak{m}\big)&\stackrel{\rm\eqref{Wprl}}{=}&\big(\F\ro\lm_x-\lm_x\ro\F\big)(y)+\F\big([x,y]_\mathfrak{m}\big)\\
&\stackrel{\rm\ref{Cpt}}{=}&-\lm_{y}\ro\F(x)+\F\big([x,y]_\mathfrak{m}\big)\\
&\stackrel{\rm\eqref{Wprl}}{=}&\F\big([y,x]_\mathfrak{m}\big)+\F\big([x,y]_\mathfrak{m}\big)\\
&=&0.
\end{eqnarray*}
Thus,  since $\F$ is a linear isomorphism, we conclude that $[\m,\m]\subseteq\h$. However, this contradicting the fact that $G/H$ is not locally symmetric.
\end{proof}

Note that, in the case where $G/H$ is locally symmetric, the Zero-One connection and the natural symplectic connection both coincide with the canonical connection of the first (and second) kind. Thus, using the fact that $\F$ is $G$-invariant, one can deduce from \cite[pp. 193-194]{Kob2} that $\F$ is parallel.

\section{The Wallash flag manifold $\SU(3)/\mathbb{T}^2$}

In this section we give some explicit examples of symplectic reductive homogeneous spaces for which the natural symplectic connection is Ricci-positive, Ricci-flat, or Ricci-negative. Thereafter, we prove our last main result.

For the positive case, it is well known that the complex projective space $\operatorname{\cplx P}^n\cong\U(n+1)/\U(n)\times\mathbb{S}^1$ is a compact Hermitian symmetric space and the Fubini-Study metric $\met_{\operatorname{FS}}$ is an $\U(n+1)$-invariant K\"ahler metric on it. Moreover, if we denote by $\Omega_0$ the K\"ahler form, we obtain that $(\operatorname{\cplx P}^n,\Omega_0)$ is a symplectic symmetric space. Thus, since $\operatorname{\cplx P}^n$ is symmetric, the natural symplectic connection coincides with the Levi-Civita connection associated to $\met_{\operatorname{FS}}$, and therefore it is Ricci positive, since $(\operatorname{\cplx P}^n,\met_{\operatorname{FS}})$ is a positive K\"ahler-Einstein manifold. On the other hand, we have seen in Proposition \ref{NilRicFlt} that the natural symplectic connection of any symplectic nilmanifold $(N/\Gamma,\Omega)$ is Ricci-flat. We refere to Examples \ref{rriF1} and \ref{rriF2} for some explicit examples of symplectic nilmanifolds. For the negative case, we consider the following Lie group
\begin{equation*}
G:=\SU(3)=\Big\{A\in\SL(3,\cplx)\,|\,\,\overline{A}^{\T}A=\I_3\Big\}.
\end{equation*}
We denote by $H$ the following closed Lie subgroup of $G$
\begin{equation*}
\mathbb{T}^2:=\left\{\begin{psmallmatrix}
e^{\mathbf{i}\theta_1}&0&0\\
0&e^{\mathbf{i}\theta_2}&0\\
0&0&e^{-\mathbf{i}(\theta_1+\theta_2)}
\end{psmallmatrix}\,|\,\,\theta_1,\theta_2\in{\mathbb{R}}\right\}.
\end{equation*}
The Lie algebras $\mathfrak{g},\mathfrak{h}$ of $G$ and $H$ respectively are given by:
\begin{eqnarray*}
\mathfrak{g}&=&\left\{\begin{psmallmatrix}
\mathbf{i}t_1&x_1+\mathbf{i}x_2&x_3+\mathbf{i}x_4\\
-x_1+\mathbf{i}x_2&\mathbf{i}t_2&x_5+\mathbf{i}x_6\\
-x_3+\mathbf{i}x_4&-x_5+\mathbf{i}x_6&-\mathbf{i}(t_1+t_2)
\end{psmallmatrix}\,|\,\,t_1,t_2,x_1,\ldots,x_6\in{\mathbb{R}}\right\};\\
\h&=&\left\{\begin{psmallmatrix}
\mathbf{i}t_1&0&0\\
0&\mathbf{i}t_2&0\\
0&0&-\mathbf{i}(t_1+t_2)
\end{psmallmatrix}\,|\,\,t_1,t_2\in{\mathbb{R}}\right\}.
\end{eqnarray*}
Moreover, the Killing form of $\mathfrak{g}$ is:
\begin{equation*}
\kappa_\mathfrak{g}:\mathfrak{g}\times\mathfrak{g}\to{\mathbb{R}},\quad\textnormal{written}\quad (X,Y)\mapsto 6\tr(XY).
\end{equation*}
Since $G$ is compact and semisimple, the $-\frac{1}{6}$ of the Killing form induces an $\Ad(G)$-invariant inner product on $\mathfrak{g}$ which will be denoted by $\met$. Let $\mathfrak{m}$ be the orthogonal complement of $\mathfrak{h}$ in $\mathfrak{g}$ with respect to $\met$. Explicitly, it is given by:
\begin{eqnarray*}
\mathfrak{m}&=&\Big\{Y\in\mathfrak{g}\,|\,\,\langle X,Y\rangle =0,\,\forall\,X\in\h\Big\}\\
&=&\Big\{Y\in\mathfrak{g}\,|\,\,\tr(XY)=0,\,\forall\,X\in\h\Big\}\\
&=&\left\{\begin{psmallmatrix}
0&x_1+\mathbf{i}x_2&x_3+\mathbf{i}x_4\\
-x_1+\mathbf{i}x_2&0&x_5+\mathbf{i}x_6\\
-x_3+\mathbf{i}x_4&-x_5+\mathbf{i}x_6&0
\end{psmallmatrix}\,|\,\,x_1,\ldots,x_6\in{\mathbb{R}}\right\}.
\end{eqnarray*}
Consider the following basis of $\mathfrak{g}$:
\begin{eqnarray*}
&&f_1:=\begin{psmallmatrix}
\mathbf{i}&0&0\\
0&0&0\\
0&0&-\mathbf{i}
\end{psmallmatrix},\, f_2:=\begin{psmallmatrix}
0&0&0\\
0&\mathbf{i}&0\\
0&0&-\mathbf{i}
\end{psmallmatrix},\, e_1:=\begin{psmallmatrix}
0&1&0\\
-1&0&0\\
0&0&0
\end{psmallmatrix},\,e_2:=\begin{psmallmatrix}
0&\mathbf{i}&0\\
\mathbf{i}&0&0\\
0&0&0
\end{psmallmatrix}\\
&& e_3:=\begin{psmallmatrix}
0&0&1\\
0&0&0\\
-1&0&0
\end{psmallmatrix},\,e_4:=\begin{psmallmatrix}
0&0&\mathbf{i}\\
0&0&0\\
\mathbf{i}&0&0
\end{psmallmatrix},\, e_5:=\begin{psmallmatrix}
0&0&0\\
0&0&1\\
0&-1&0
\end{psmallmatrix},\,e_6:=\begin{psmallmatrix}
0&0&0\\
0&0&\mathbf{i}\\
0&\mathbf{i}&0
\end{psmallmatrix}.
\end{eqnarray*}
Clearly that $\h=\Span_{\mathbb{R}}\{f_1,f_2\}$ and $\mathfrak{m}=\Span_{\mathbb{R}}\{e_1,\ldots,e_6\}$. Furthermore, the non vanishing Lie brackets in this basis are:
\begin{eqnarray*}
&&[f_2,e_2]=-[f_1,e_2]=-[e_3,e_5]=-[e_4,e_6]=e_1,\\
&&[f_1,e_1]=-[f_2,e_1]=[e_3,e_6]=-[e_4,e_5]=e_2,\\
&&[e_1,e_5]=-[f_2,e_4]=-\tfrac{1}{2}[f_1,e_4]=-[e_2,e_6]=e_3,\\
&&[f_2,e_3]=\tfrac{1}{2}[f_1,e_3]=[e_1,e_6]=[e_2,e_5]=e_4,\\
&&[f_1,e_6]=\tfrac{1}{2}[f_2,e_6]=[e_1,e_3]=[e_2,e_4]=-e_5,\\
&&[f_1,e_5]=\tfrac{1}{2}[f_2,e_5]=-[e_1,e_4]=[e_2,e_3]=e_6,\\
&&[e_1,e_2]=2f_1-2f_2,\,\,[e_3,e_4]=2f_1,\,\,[e_5,e_6]=2f_2.
\end{eqnarray*}
For each $X=(x_1,\ldots,x_6)\in\mathfrak{m}$, the matrix representation of the linear endomorphism $\lm_X:\mathfrak{m}\to\mathfrak{m},\,Y\mapsto[X,Y]_\mathfrak{m}$ in the basis $\{e_k\}_{1\leqslant k\leqslant 6}$ is given by:
\begin{equation}\label{admXSU3}
\begin{pmatrix}
0&0&\Ac&\Bc&-\Ab&-\Bb\\
0&0&-\Bc&\Ac&-\Bb&\Ab\\
-\Ac&\Bc&0&0&\Aa&-\Ba\\
-\Bc&-\Ac&0&0&\Ba&\Aa\\
\Ab&\Bb&-\Aa&-\Ba&0&0\\
\Bb&-\Ab&\Ba&-\Aa&0&0
\end{pmatrix}.
\end{equation}
Similarly, for $Y=(\Ca,\ldots,\Dc)\in\mathfrak{m}$ the matrix representation of the linear endomorphism $D_{X,Y}:\mathfrak{m}\to\mathfrak{m},\,Z\mapsto\left[[X,Z]_\h,Y\right]$ in the basis $\{e_k\}_{1\leqslant k\leqslant 6}$ is:
\begin{equation*}
\left(\begin{array}{cccccc} 4\Ba\Da&-4\Aa\Da&2\Bb\Da&-2\Ab\Da&-2\Bc\Da&2\Ac\Da\\ 
\noalign{\medskip}-4\Ba\Ca&4\Aa\Ca&-2\Bb\Ca&2\Ab\Ca&2\Bc\Ca&-2\Ac\Ca\\
\noalign{\medskip}2\Ba\Db&-2\Aa\Db&4\Bb\Db&-4\Ab\Db&2\Bc\Db&-2\Ac\Db\\
\noalign{\medskip}-2\Ba\Cb&2\Aa\Cb&-4\Bb\Cb&4\Ab\Cb&-2
\Bc\Cb&2\Ac\Cb\\ 
\noalign{\medskip}-2\Ba\Dc&2\Aa\Dc&2\Bb\Dc&-2\Ab\Dc&4
\Bc\Dc&-4\Ac\Dc\\
\noalign{\medskip}2\Ba\Cc&-2\Aa\Cc&-2\Bb\Cc&2\Ab\Cc&-4\Bc\Cc&4\Ac\Cc\end{array}\right).
\end{equation*}
In particular, we have:
\begin{equation}\label{RicSPartI}
\tr\big(\lm_X\ro\lm_Y\big)=-\tr\big(D_{X,Y}\big)=-4\sum_{k=1}^6x_ky_k.
\end{equation}
Let $X_0:=\begin{psmallmatrix}
\mathbf{i}&0&0\\
0&-\mathbf{i}&0\\
0&0&0
\end{psmallmatrix}\in\h$, and consider the following symplectic tensor of $\mathfrak{m}$ defined by:
\begin{eqnarray*}
\w(X,Y)&:=&\tfrac{1}{6}\,\kappa_\mathfrak{g}\left(X_0,[Y,X]\right)\\
&=&\tr(XX_0Y-YX_0X)\\
&=&4(x_1y_2-x_2y_1)+2(x_3y_4-x_4y_3)-2(x_5y_6-x_6y_5),
\end{eqnarray*}
for $X=(x_1,\ldots,x_6),Y=(y_1,\ldots,y_6)\in\mathfrak{m}$. A straightforward verification shows that $\omega$ defines an $\SU(3)$-invariant symplectic structure on the Wallash flag manifold $\SU(3)/\mathbb{T}^2$. Further, the linear isomorphism $\F:\mathfrak{m}\to\mathfrak{m}$ defined in \eqref{W} that relates $\omega$ with $\met$ is equal to $\ad_{X_0}:\mathfrak{m}\to\mathfrak{m}$, and its matrix representation in the basis $\{e_k\}_{1\leqslant k\leqslant 6}$ is:
\begin{equation}\label{WSU3}
\begin{pmatrix}
0&-2&0&0&0&0\\
2&0&0&0&0&0\\
0&0&0&-1&0&0\\
0&0&1&0&0&0\\
0&0&0&0&0&1\\
0&0&0&0&-1&0
\end{pmatrix}.
\end{equation}
The Zero-One connection $\nc$ is defined by the following product
\begin{eqnarray*}
\omega\left(\la_X(Y),Z\right)&:=&\omega\left([X,Z]_\mathfrak{m},Y\right),
\end{eqnarray*}
for $X,Y,Z\in\mathfrak{m}$. To give an explicit formula for this product, we will use the fourth assertion in Proposition \ref{Cpt}. For $X=(x_1,\ldots,x_6)\in\mathfrak{m}$, using \eqref{admXSU3} and \eqref{WSU3}, we get that the matrix representation of the linear endomorphism $\la_{X}:\mathfrak{m}\to\mathfrak{m}$ in the basis $\{e_k\}_{1\leqslant k\leqslant 6}$ is:
\begin{equation}\label{LcSU3}
\begin{pmatrix}
0&0&\tfrac{1}{2}x_5&\tfrac{1}{2}x_6&-\tfrac{1}{2}x_3&-\tfrac{1}{2}x_4\\
0&0&-\tfrac{1}{2}x_6&\tfrac{1}{2}x_5&-\tfrac{1}{2}x_4&\tfrac{1}{2}x_3\\
-2x_5&2x_6&0&0&-x_1&x_2\\
-2x_6&-2x_5&0&0&-x_2&-x_1\\
2x_3&2x_4&x_1&x_2&0&0\\
2x_4&-2x_3&-x_2&x_1&0&0
\end{pmatrix}.
\end{equation}
Analogously, the natural symplectic connection $\ns$ is defined by the following product
\begin{eqnarray*}
\ls_X(Y)&:=&\tfrac{1}{3}\big\{\la_X(Y)+\lm_X(Y)\big\},
\end{eqnarray*}
for $X,Y\in\mathfrak{m}$. Hence, using \eqref{admXSU3} and \eqref{LcSU3}, we obtain that the matrix representation of the linear endomorphism $\ls_X:\mathfrak{m}\to\mathfrak{m}$ in the basis $\{e_k\}_{1\leqslant k\leqslant 6}$ is: 
\begin{equation}\label{LsSU3}
\begin{pmatrix}
0&0&\tfrac{1}{2}x_5&\tfrac{1}{2}x_6&-\tfrac{1}{2}x_3&-\tfrac{1}{2}x_4\\
0&0&-\tfrac{1}{2}x_6&\tfrac{1}{2}x_5&-\tfrac{1}{2}x_4&\tfrac{1}{2}x_3\\
-x_5&x_6&0&0&0&0\\
-x_6&-x_5&0&0&0&0\\
x_3&x_4&0&0&0&0\\
x_4&-x_3&0&0&0&0
\end{pmatrix},
\end{equation}
for $X=(x_1,\ldots,x_6)\in\mathfrak{m}$. To compute the Ricci curvature of $\ns$, we will use the formula \eqref{RicSCpt}. Let $X=(\Aa,\ldots,\Bc),\, Y=(\Ca,\ldots,\Dc)\in\mathfrak{m}$, since the matrix representation of the linear endomorphism $\lm_{\ad_{X_0}(Y)}:\mathfrak{m}\to\mathfrak{m}$ in the basis $\{e_k\}_{1\leqslant k\leqslant 6}$ is:
\begin{equation*}
\begin{pmatrix}
0&0&y_6&-y_5&y_4&-y_3\\
0&0&y_5&y_6&-y_3&-y_4\\
-y_6&-y_5&0&0&-2 y_2&-2 y_1\\
y_5&-y_6&0&0&2 y_1&-2 y_2\\
-y_4&y_3&2 y_2&-2 y_1&0&0\\
y_3&y_4&2 y_1&2 y_2&0&0\\
\end{pmatrix}.
\end{equation*}
Consequently,
\begin{equation}\label{RicSPartII}
\tr\left(\lm_X\ro\ad_{X_0}^{-1}\ro\lm_{\ad_{X_0}(Y)}\right)=-8\sum_{k=1}^2x_ky_k+\sum_{k=3}^6x_ky_k.
\end{equation}
Therefore, using \eqref{RicSPartI} and \eqref{RicSPartII}, we get:
\begin{equation*}
\ric^\mathbf{s}(X,Y)=-4\Big\{2\big(\Aa\Ca+\Ba\Da\big)+\Ab\Cb+\Bb\Db+\Ac\Cc+\Bc\Dc\Big\}.
\end{equation*}

Now, we will prove Theorem \ref{SU3Pref}. First, let us show that $\ns$ is Ricci-parallel, i.e., $\ns\ric^\mathbf{s}=0$. Since $\ric^\mathbf{s}$ is $\SU(3)$-invariant, by using $\mathcal{L}_{X^\#}\ric^\mathbf{s}=0$, one can easily check that $\ns$ is Ricci-parallel if and only if
\begin{equation*}\label{PerfSU3-2-}
\ric^\mathbf{s}\big(\ls_X(Y),Z\big)=-\ric^\mathbf{s}\big(\ls_X(Z),Y\big).
\end{equation*}
Let $X=(x_1,\ldots,x_6),Y=(y_1,\ldots,y_6),$ and $Z=(z_1,\ldots,z_6)\in\mathfrak{m}$. By using \eqref{LsSU3}, we obtain:
\begin{eqnarray*}
\ls_X(Y)&=&\Big(\tfrac{1}{2}\big\{x_5y_3-x_3y_5+x_6y_4-x_4y_6\big\},\,\tfrac{1}{2}\big\{x_5y_4-x_4y_5+x_3y_6-x_6y_3\big\},\\
&&x_6y_2-x_5y_1,\,-x_5y_2-x_6y_1,\,x_3y_1+x_4y_2,\,x_4y_1-x_3y_2\Big).
\end{eqnarray*}
Thus,
\begin{eqnarray*}
\tfrac{1}{4}\ric^\mathbf{s}\big(\ls_X(Y),Z\big)&=&x_3\big\{y_2z_6-y_6z_2+y_5z_1-y_1z_5\big\}+x_4\big\{y_5z_2-y_2z_5+y_6z_1-y_1z_6\big\}\\
&&+x_5\big\{y_1z_3-y_3z_1+y_2z_4-y_4z_2\big\}+x_6\big\{y_3z_2-y_2z_3+y_1z_4-y_4z_1\big\}\\
&=&-\tfrac{1}{4}\ric^\mathbf{s}\big(\ls_X(Z),Y\big).
\end{eqnarray*}
It follows that the natural symplectic connection $\ns$ of the Wallash flag manifold $\SU(3)/\mathbb{T}^2$ is Ricci-parallel. Hence, Theorem \ref{SU3Pref} results from the uniqueness of the preferred symplectic connection of $\SU(3)/\mathbb{T}^2$, and the computations above.\\

We end this paper by noting that similar computations can be carried out for the flag manifold $\SU(4)/\mathbb{T}^3$. However, one can check that the natural symplectic connection $\ns$ is not preferred.


\end{document}